\documentclass[11pt]{amsart}
\usepackage{amsopn}
\usepackage{amssymb, amscd}
\usepackage{amsmath}
\usepackage[table]{xcolor}

\usepackage{tikz}
\usetikzlibrary{matrix,positioning}
\usetikzlibrary{matrix,fit}
\usepackage{txfonts}
\usepackage{tikz}
\usetikzlibrary{matrix,backgrounds,fit}
\usepackage{mleftright}

\usepackage{MnSymbol}
\usetikzlibrary{matrix}

\usepackage{tikz}

\newcommand{\nc}{\newcommand}

\nc{\eg}{\mathfrak{e} } \nc{\fg}{\mathfrak{f} } \nc{\vg}{\mathfrak{v} } \nc{\wg}{\mathfrak{w} }
\nc{\zg}{\mathfrak{z} } \nc{\ngo}{\mathfrak{n} } \nc{\kg}{\mathfrak{k} }
\nc{\mg}{\mathfrak{m} } \nc{\bg}{\mathfrak{b} } \nc{\ggo}{\mathfrak{g} }
\nc{\ggob}{\overline{\mathfrak{g}} } \nc{\sog}{\mathfrak{so} }
\nc{\sug}{\mathfrak{su} } \nc{\spg}{\mathfrak{sp} } \nc{\slg}{\mathfrak{sl} }
\nc{\glg}{\mathfrak{gl} } \nc{\cg}{\mathfrak{c} } \nc{\rg}{\mathfrak{r} }
\nc{\hg}{\mathfrak{h} } \nc{\tg}{\mathfrak{t} } \nc{\ug}{\mathfrak{u} }
\nc{\dg}{\mathfrak{d} } \nc{\ag}{\mathfrak{a} } \nc{\pg}{\mathfrak{p} }
\nc{\sg}{\mathfrak{s} } \nc{\affg}{\mathfrak{aff} }

\nc{\pca}{\mathcal{P}} \nc{\nca}{\mathcal{N}} \nc{\lca}{\mathcal{L}}
\nc{\oca}{\mathcal{O}} \nc{\mca}{\mathcal{M}} \nc{\tca}{\mathcal{T}}
\nc{\aca}{\mathcal{A}} \nc{\cca}{\mathcal{C}} \nc{\gca}{\mathcal{G}}
\nc{\sca}{\mathcal{S}} \nc{\hca}{\mathcal{H}} \nc{\bca}{\mathcal{B}}
\nc{\dca}{\mathcal{D}} \nc{\zca}{\mathcal{Z}}

\nc{\val}{\operatorname{val}}
\nc{\vp}{\varphi} \nc{\ddt}{\tfrac{d}{dt}} \nc{\dds}{\tfrac{{\rm d}}{{\rm d}s}}
\nc{\dpar}{\tfrac{\partial}{\partial t}} \nc{\im}{\mathtt{i}}

\nc{\SO}{\mathrm{SO}} \nc{\Spe}{\mathrm{Sp}} \nc{\Sl}{\mathrm{SL}}
\nc{\SU}{\mathrm{SU}} \nc{\Or}{\mathrm{O}} \nc{\U}{\mathrm{U}} \nc{\Gl}{\mathrm{GL}}
\nc{\Se}{\mathrm{S}} \nc{\Cl}{\mathrm{Cl}} \nc{\Spein}{\mathrm{Spin}}
\nc{\Pin}{\mathrm{Pin}} \nc{\G}{\mathrm{GL}_n(\RR)} \nc{\g}{\mathfrak{gl}_n(\RR)}

\nc{\RR}{{\Bbb R}} \nc{\HH}{{\Bbb H}} \nc{\CC}{{\Bbb C}} \nc{\ZZ}{{\Bbb Z}}
\nc{\FF}{{\Bbb F}} \nc{\NN}{{\Bbb N}} \nc{\QQ}{{\Bbb Q}} \nc{\PP}{{\Bbb P}}

\nc{\vs}{\vspace{.2cm}} \nc{\vsp}{\vspace{1cm}} \nc{\ip}{\langle\cdot,\cdot\rangle}
\nc{\ipp}{(\cdot,\cdot)} \nc{\la}{\langle} \nc{\ra}{\rangle} \nc{\unm}{\tfrac{1}{2}}
\nc{\unc}{\tfrac{1}{4}} \nc{\und}{\tfrac{1}{16}} \nc{\no}{\vs\noindent}
\nc{\lamkn}{\Lambda^2(\RR^{q+n})^*\otimes\RR^{q+n}} \nc{\lamn}{\Lambda^2(\RR^n)^*\otimes\RR^n} \nc{\lamp}{\Lambda^2\pg^*\otimes\pg}
\nc{\lamg}{\Lambda^2\ggo^*\otimes\ggo} \nc{\lamngo}{\Lambda^2\ngo^*\otimes\ngo}
\nc{\tangz}{{\rm T}^{\rm Zar}} \nc{\mum}{/\!\!/} \nc{\kir}{/\!\!/\!\!/}
\nc{\Ri}{\tfrac{4\Ricci_{\mu}}{||\mu||^2}} \nc{\ds}{\displaystyle}
\nc{\ben}{\begin{enumerate}} \nc{\een}{\end{enumerate}} \nc{\f}{\frac}
\nc{\lb}{[\cdot,\cdot]} \nc{\isn}{\tfrac{1}{||v||^2}}
\nc{\gkp}{(\ggo=\kg\oplus\pg,\ip)} \nc{\ukh}{(\ug=\kg\oplus\hg,\ip)}

\nc{\Hess}{\operatorname{Hess}} \nc{\ad}{\operatorname{ad}}
\nc{\Ad}{\operatorname{Ad}} \nc{\rank}{\operatorname{rank}}
\nc{\Irr}{\operatorname{Irr}} \nc{\End}{\operatorname{End}}
\nc{\Aut}{\operatorname{Aut}} \nc{\Inn}{\operatorname{Inn}}
\nc{\Der}{\operatorname{Der}} \nc{\Ker}{\operatorname{Ker}}
\nc{\Iso}{\operatorname{I}} \nc{\Diff}{\operatorname{Diff}}
\nc{\Lie}{\operatorname{Lie}} \nc{\tr}{\operatorname{tr}} \nc{\dif}{\operatorname{d}}
\nc{\sen}{\operatorname{sen}} \nc{\modu}{\operatorname{mod}}
\nc{\Riem}{\operatorname{Rm}} \nc{\Ricci}{\operatorname{Ric}}
\nc{\sym}{\operatorname{sym}} \nc{\symac}{\operatorname{sym^{ac}}}
\nc{\symc}{\operatorname{sym^{c}}} \nc{\scalar}{\operatorname{R}}
\nc{\grad}{\operatorname{grad}} \nc{\ricci}{\operatorname{Rc}}
\nc{\nr}{\operatorname{nr}} \nc{\riccic}{\operatorname{ric^{c}}}
\nc{\riccig}{\operatorname{ric^{\gamma}}} \nc{\Rin}{\operatorname{M}}
\nc{\Le}{\operatorname{L}} \nc{\tang}{\operatorname{T}}
\nc{\level}{\operatorname{level}} \nc{\rad}{\operatorname{r}}
\nc{\abel}{\operatorname{ab}} \nc{\CH}{\operatorname{CH}}
\nc{\mcc}{\operatorname{mcc}} \nc{\Adj}{\operatorname{Adj}}
\nc{\Order}{\operatorname{O}} \nc{\mm}{\operatorname{M}}
\nc{\inj}{\operatorname{inj}}  \nc{\Pf}{\operatorname{Pf}}   \nc{\pf}{\operatorname{pf}}
\nc{\vol}{\operatorname{vol}} \nc{\Diag}{\operatorname{Diag}}

\theoremstyle{plain}
\newtheorem{theorem}{Theorem}[section]

\theoremstyle{definition}
\newtheorem{definition}[theorem]{Definition}

\theoremstyle{remark}
\newtheorem{remark}[theorem]{Remark}

\newtheorem{example}[theorem]{Example}

\newcommand{\variable}{\Theta}
\newcommand{\variabledos}{\Phi}
\newcommand{\espacio}{\:}
\newcommand{\varuno}{x}
\newcommand{\vardos}{y}
\newcommand{\vartres}{z}

\title[Explicit linear pfaffian representations]{Explicit linear pfaffian representations\\ of plane curves up to degree 5 }
\author{David Oscari}
\address{FaMAF and CIEM, Universidad Nacional de C\'ordoba, C\'ordoba, Argentina}
\email{oscari@famaf.unc.edu.ar}
\thanks{This research was partially supported by grants from CONICET, FonCyT (Argentina) and SeCyT (Universidad Nacional de C\'ordoba)}

\begin{document}

\maketitle
\begin{abstract}Let $R$ be a commutative ring with $1$. We prove that every homogeneous polynomial $f(x_0,x_1,x_2)$ in $R[x_0,x_1,x_2]$ up to degree 5 admits a
linear Pfaffian $R$-representation. We  believe that conceptually we give the shortest self-contained proof possible: we exhibit explicitly such a
representation. In this sense, we generalize (up to degree 5) a result due to A. Beauville \cite{Bea} about the existence of Pfaffian representations for any smooth
plane curve of any degree. % with coefficients in a field algebraically closed.
%  , exhibiting
\end{abstract}

%%%%%%%%%%%%%%%%%%%%%%%%%%%%%%%%%%%%%%%%%%%%%%%%%%%%%%%%%%%%%%%%%%%%%%%%%%%%%%%%%%%%%%%%%%%%%%%%%%%
%%%%%%%%%%%%%%%%%%%%%%%%%%%%%%%%%%%%%%%%%%%%%%%%%%%%%%%%%%%%%%%%%%%%%%%%%%%%%%%%%%%%%%%%%%%%%%%%%%%
%%%%%%%%%%%%%%%%%%%%%%%%%%%%%%%%%%%%%%%%%     SECTION                     %%%%%%%%%%%%%%%%%%%%%%%%%
%%%%%%%%%%%%%%%%%%%%%%%%%%%%%%%%%%%%%%%%%%%%%%%%%%%%%%%%%%%%%%%%%%%%%%%%%%%%%%%%%%%%%%%%%%%%%%%%%%%
%%%%%%%%%%%%%%%%%%%%%%%%%%%%%%%%%%%%%%%%%%%%%%%%%%%%%%%%%%%%%%%%%%%%%%%%%%%%%%%%%%%%%%%%%%%%%%%%%%%

 \section{Introduction}
 This note can be considered as a derivation of \cite{LO}, in collaboration with J. Lauret, where we explore nonsingular 2-step nilpotent Lie algebras
 $\ngo=\ngo_1\oplus[\ngo,\ngo]$ over $\mathbb{R}$ in several directions, among them we consider the problem of existence and non-existence of canonical inner
 products (Einstein nilsolitons) for such algebras of type $(p,q)=(3,8)$, i.e. $\dim [\ngo,\ngo]=p$ and $\dim \ngo_1=q$. In {\it loc. cit.}, we exhibit explicit
 continuous families of pairwise non-isomorphic nonsingular algebras.
 %Propositions 4.1, 4.2 and 4.3 (existence of Einstein nilsolitons) and Propositions 4.4 and 4.5 (non-existence of Einstein nilsolitons).

To determine if two given 2-step nilpotent Lie algebras are isomorphic or not is a hard \mbox{problem.} For that, we associate to each real $2$-step nilpotent Lie
algebra of type $(p,q)=(3,8)$ its {\it Paffian form} (homogeneous polynomial in $\mathbb{R}[x,y,z]$ of degree 4) which has been a very useful tool to distinguish these
algebras up to isomorphism.

A question that arises naturally in this context is the following. Inversely, for a given homogeneous polynomial $f$ in $\mathbb{R}[x,y,z]$ of degree 4, to find, if it
exists, a $2$-step nilpotent Lie algebra over $\mathbb{R}$ of type $(3,8)$ such that its Pfaffian form is  $f(x,y,z)$. %If $f$ is {\it positive} (i.e. $f(x,y,z)>0$ for all nonzero $(x,y,z)$ in $\mathbb{R}^3$), a such 2-step Lie algebra exists (Ph.D. Tesis of the author).

Indeed, this is a theoretical-Lie algebra version of a classical problem:
\begin{equation}\label{problema}
    \begin{minipage}{0.9\textwidth}
    Let $R$ be a commutative ring with $1$. For a given homogeneous polynomial $f(x_0,\dots,x_n)$ in $R[x_0,\dots,x_n]$ of degree $d$, determine if there exists a
    linear Pfaffian $R$-representation of $f(x_0,\dots,x_n)$.  %where $R$ is a commutative ring with $1\in R$.
    \end{minipage}\tag{$\ast$}
  \end{equation}
Let us
recall that a linear Pfaffian $R$-representation of a homogeneous polynomial $f$ in $R[x_0,\dots,x_n]$ of degree $d$ is a $2d\times2d$ skew-symmetric matrix $M=[m_{ij}]$
such that $\det(M)=f(x_0,\dots,x_n)^2$ where the entries $m_{ij}$ are linear forms in $R[x_0,\dots,x_n]$ (see Section 2 for more precise definitions).

%%%    \textcolor{red}{Let $\mathbb{K}$ be a field algebraically closed of characteristic zero.}

%\begin{quote}  Determine existence of a linear Pfaffian $R$-representation for a give homogeneous polynomial of degree $d$ in $R[x,y,z]$ where $R$ is a commutative ring with $1\in R$.  \end{quote}
In the case $n=2$, Beauville showed, among others, that (\ref{problema}) has solution for any smooth plane curve in $\mathbb{K}[x,y,z]$ of degree $d\geq 2$
\cite[Proposition 5.1]{Bea}, where  $\mathbb{K}$ is an algebraically closed field of characteristic zero. Buckley and K\v{o}sir \cite{BK} parametrise all linear Pfaffian representations of a plane curve. Two Pfaffian $\mathbb{K}$-representations
$M$ and $M'$ are equivalent if there exists $X\in\text{GL}_{2d}(\mathbb{K})$ such that $M'=XMX^t$.
In \cite{B} the elementary transformations of linear Pfaffian representations are considered, proving that every two Pfaffian representations of a plane curve
$f(x_0,x_1,x_2)$ of degree $d$ can be bridged by a finite sequence of elementary transformations.

%%%%%%%%%%%%%%%%       meter
%In the case $n=3$, Beauville showed that every smooth cubic surface admits a linear Pfaffian representation \cite[Proposition 7.6(a)]{Bea}, and the fact
%that a {\it general} surface of degree $d$ admits a linear Pfaffian  representation if and only if $d\leq15$ had been proved by Beauville-Schreyer
%\cite[Proposition 7.6(b)]{Bea}. In \cite{CKM} a proof of that every smooth quartic surface admits a linear Pfaffian representation is provided.

In the case $n=3$, Beauville showed that every smooth cubic surface admits a linear Pfaffian representation \cite[Proposition 7.6(a)]{Bea}. Then in \cite{FM} is proved
that every cubic surface admits such a linear Pfaffian representation.

The fact that a {\it general} surface of degree $d$ admits a linear Pfaffian  representation if and only if $d\leq15$ had been proved by Beauville-Schreyer
\cite[Proposition 7.6(b)]{Bea}. When the degree $d=4$ in \cite{CKM} the previous result is generalized: every smooth quartic surface admits a linear
Pfaffian representation. Faenzi \cite{F} proved that a general surface of degree $d$ admits a almost quadratic Pfaffian representation if and only if $d\leq15$.
In \cite{CF} the Pfaffian representations with entries almost linear of general surfaces  are considered.

%In \cite[Appendix V]{AR} Adler proved that a {\it generic} cubic threefold admits a Pfaffian representation, and in \cite{IM} is proved that a {\it generic}
%quartic threefold admits a Pfaffian representation.  The method used in both articles is to prove that the differential of the Pfaffian map is of maximal rank
%in a particular skew-symmetric matrix $M_0$ with linear entries. Beauville proved that every smooth cubic threefold admits a Pfaffian representation \cite[Proposition 8.5]{Bea}.

The author was motivated by the scarcity of explicit Pfaffian representations in the literature.

Independently, Tanturri and Han were interested on the explicit construction of linear Pfaffian representations for cubic surfaces.
In \cite{T} an algorithm is provided whose inputs are a cubic surface $f(x,y,z,t)$ in $\mathbb{K}[x,y,z,t]$  and a zero ${\mathbf a}$ in $\mathbb{P}_\mathbb{K}^3$
and whose output is a linear Pfaffian $\mathbb{K}$-representation of $f$, under assumptions on $f$ and $\mathbf{a}$, and where $\mathbb{K}$ is a field of characteristic zero,
not necessarily algebraically closed. In \cite{H} first one chooses an inscribed pentahedron and, then one needs to find the unique Pfaffian bundle related to that pentahedron.

We focus on the case $n=2$ and degrees $d=2,3,4$, and $5$. Our main result is the following.
%bridging two pfaffian representations by a finite sequence of elementary transformations.
%The proofs use Algebraic Geometry methods.

%%%%%%         We answer affirmatively to the problem (\ref{problema}) for any plane curve of degrees $d\leq 5$.
%which improves the known results (for small degrees) about the existence of Pfaffian representations \cite{Bea}.

\begin{theorem}\label{Grado 4 y 5}
Let $R$ be a commutative ring with $1$. Every homogeneous polynomial $f(x,y,z)$ in $R[x,y,z]$ of degree $\leq 5$ admits a linear Pfaffian $R$-representation.
\end{theorem}

Section 2 is devoted to define precisely the Pfaffian  of a skew-symmetric matrix and a linear Pfaffian representation, and for completeness we solve the problem
(\ref{problema}) for degrees $d\leq3$ (Example \ref{ejemplo: grado 2 y 3}). Finally, in Section 3 we give explicit linear Pfaffian representations for arbitrary
forms of degree $4$ and $5$.\\

\noindent {\it Acknowledgements.}  I am very grateful to Agust\'{\i}n Garc\'{\i}a Iglesias for providing useful comments on a draft version of this article and for suggest
the Definition \ref{nice}. I wish to thank David Eisenbud for responding to my inquiries about the main result.

%%%%%%%%%%%%%%%%%%%%%%%%%%%%%%%%%%%%%%%%%%%%%%%%%%%%%%%%%%%%%%%%%%%%%%%%%%%%%%%%%%%%%%%%%%%%%%%%%%%
%%%%%%%%%%%%%%%%%%%%%%%%%%%%%%%%%%%%%%%%%%%%%%%%%%%%%%%%%%%%%%%%%%%%%%%%%%%%%%%%%%%%%%%%%%%%%%%%%%%
%%%%%%%%%%%%%%%%%%%%%%%%%%%%%%%%%%%%%%%%%     SECTION                     %%%%%%%%%%%%%%%%%%%%%%%%%
%%%%%%%%%%%%%%%%%%%%%%%%%%%%%%%%%%%%%%%%%%%%%%%%%%%%%%%%%%%%%%%%%%%%%%%%%%%%%%%%%%%%%%%%%%%%%%%%%%%
%%%%%%%%%%%%%%%%%%%%%%%%%%%%%%%%%%%%%%%%%%%%%%%%%%%%%%%%%%%%%%%%%%%%%%%%%%%%%%%%%%%%%%%%%%%%%%%%%%%

\section{Pfaffian and linear Pfaffian representations}

There are several equivalent definitions of the Pfaffian  of a skew-symmetric matrix. In our computations with {\small\textsf{Maple\texttrademark}}, we have used a
Laplace-type expansion. For other definitions of Pfaffian and its properties, we refer the reader to \cite{DW}, \cite[Appendix D]{FP}. For a history of the Pfaffians
we recommend \cite{K}.

\begin{definition}[Expansion along row 1]\label{definicion recursiva}
The {\it Pfaffian} $\Pf(A)$ of a $2d\times2d$ skew-symmetric matrix $A=[a_{ij}]$ with entries in a commutative ring $R$ with $1$ is defined recursively as
\begin{align*}
 \Pf\left(\left(
  \begin{array}{rrr}
    0 & a_{12} \\
    -a_{12} & 0
  \end{array}
  \right)\right) & :=   a_{12}  & \textrm{ if }d=1\,, \\
\Pf(A)   & :=  \sum_{j=2}^{2d}(-1)^j\cdot a_{1j}\cdot \textrm{Pf}\left(A^{[1,j]}\right)   &  \textrm{ if }d\geq2\,,
\end{align*}
where $A^{[1,j]}$ is the $(2d-2)\times(2d-2)$-matrix obtained from $A$ by deleting both the $(1,j)$-th rows and $(1,j)$-th columns.
\end{definition}

\begin{example}\label{ejemplo: grado 2 y 3}Fix $\variable_1,\dots,\variable_{10}\in R$. Let $M_2=\varuno\, [a_{ij}]+\vardos\, [b_{ij}]+\vartres\,
[c_{ij}]$ and $M_3=\varuno\, [\tilde{a}_{ij}]+\vardos\, [\tilde{b}_{ij}]+\vartres\, [\tilde{c}_{ij}]$ be the $4\times4$ skew-symmetric matrix,  respectively
the $6\times6$ skew-symmetric matrix, where

\begin{equation}\label{m_2: a,b,c}
[a_{ij}]=  \left[
\begin {smallmatrix}
%%%%%%%%%%%%%%%%%%%%%%%%%%%%%
0&1&\Theta_{4}&\Theta_{5}\\
\noalign{\medskip}&0&0&0\\
\noalign{\medskip}&&0&\Theta_{1}\\
\noalign{\medskip}*&&&0
%%%%%%%%%%%%%%%%%%%%%%%%%%%%%
\end {smallmatrix} \right]\,,
\quad
[b_{ij}]=  \left[
\begin {smallmatrix}
%%%%%%%%%%%%%%%%%%%%%%%%%%%%%
0&0&\Theta_{2}&\Theta_{6}\\
\noalign{\medskip}&0&0&-1\\
\noalign{\medskip}&&0&0\\
\noalign{\medskip}*&&&0
%%%%%%%%%%%%%%%%%%%%%%%%%%%%%
\end {smallmatrix} \right]\,,
\quad
[b_{ij}]=  \left[
\begin {smallmatrix}
%%%%%%%%%%%%%%%%%%%%%%%%%%%%%
0&0&0&\Theta_{3}\\
\noalign{\medskip}&0&1&0\\
\noalign{\medskip}&&0&0\\
\noalign{\medskip}*&&&0
%%%%%%%%%%%%%%%%%%%%%%%%%%%%%
\end {smallmatrix} \right]\,,
\end{equation}

\begin{equation}\label{m_3: a,b,c}
[\tilde{a}_{ij}]=  \left[
\begin {smallmatrix}
%%%%%%%%%%%%%%%%%%%%%%%%%%%%%
0&\Theta_{1}&\Theta_{10}&0&\Theta_{6}&\Theta_{5}\\
\noalign{\medskip}&0&0&0&0&0\\
\noalign{\medskip}&&0&-1&0&0\\
\noalign{\medskip}&&&0&0&0\\
\noalign{\medskip}&*&&&0&-1\\
\noalign{\medskip}&&&&&0
%%%%%%%%%%%%%%%%%%%%%%%%%%%%%
\end {smallmatrix} \right]\,,
\quad
[\tilde{b}_{ij}]=  \left[
\begin {smallmatrix}
%%%%%%%%%%%%%%%%%%%%%%%%%%%%%
0&\Theta_{{4}}&0&0&-1&\Theta_{{2}}\\
\noalign{\medskip}&0&-1&0&0&0\\
\noalign{\medskip}&&0&0&0&0\\
\noalign{\medskip}&&&0&-1&0\\
\noalign{\medskip}&*&&&0&0\\
\noalign{\medskip}&&&&&0
%%%%%%%%%%%%%%%%%%%%%%%%%%%%%
\end {smallmatrix} \right]\,,
\quad
[\tilde{c}_{ij}]=  \left[
\begin {smallmatrix}
%%%%%%%%%%%%%%%%%%%%%%%%%%%%%
0&0&\Theta_{{9}}&1&\Theta_{{7}}&\Theta_{{8}}\\
\noalign{\medskip}&0&0&0&0&1\\
\noalign{\medskip}&&0&0&\Theta_{{3}}&0\\
\noalign{\medskip}&&&0&0&0\\
\noalign{\medskip}&*&&&0&0\\
\noalign{\medskip}&&&&&0
%%%%%%%%%%%%%%%%%%%%%%%%%%%%%
\end {smallmatrix} \right]\,.
\end{equation}
Applying Definition \ref{definicion recursiva}, we obtain
\begin{align*}
\Pf(M_2)&=\variable_1\espacio x^2+\variable_2\espacio y^2+\variable_3\espacio z^2+\variable_4\espacio xy+\variable_5\espacio xz+\variable_6\espacio yz  \textrm{\,,}\\
\Pf(M_3)&=\variable_1\espacio x^3+\variable_2\espacio y^3+\variable_3\espacio z^3+\variable_4\espacio x^2y+\variable_5\espacio xy^2+\variable_6\espacio x^2z+\variable_7
\espacio xz^2+ \variable_8\espacio y^2z+\variable_9\espacio yz^2+\variable_{10}\espacio xyz\,.
\end{align*}

\end{example}

%%%%%%%%%%%%%%%%%%%%%%%%%%%%%%%%%%%%%%%%%%%%%%%%%%%%%%%%%%%%%%%%%%%%%%%%%%%%%%%%%%%%%%%%%%%%%%%%%%%
%%%%%%%%%%%%%%%%%%%%%%%%%%%%%%%%%%%%%%%%%%%%%%%%%%%%%%%%%%%%%%%%%%%%%%%%%%%%%%%%%%%%%%%%%%%%%%%%%%%
%%%%%%%%%%%%%%%%%%%%%%%%%%%%%%%%%%%%%%%%%     SECTION                     %%%%%%%%%%%%%%%%%%%%%%%%%
%%%%%%%%%%%%%%%%%%%%%%%%%%%%%%%%%%%%%%%%%%%%%%%%%%%%%%%%%%%%%%%%%%%%%%%%%%%%%%%%%%%%%%%%%%%%%%%%%%%
%%%%%%%%%%%%%%%%%%%%%%%%%%%%%%%%%%%%%%%%%%%%%%%%%%%%%%%%%%%%%%%%%%%%%%%%%%%%%%%%%%%%%%%%%%%%%%%%%%%

\subsection{Linear Pfaffian representation}

\begin{definition}\label{representation pfaffiana linear} Let $R$ be a commutative ring with $1$ and let $f(x_0,\dots,x_n)$ be a homogeneous polynomial in
$R[x_0,\dots,x_n]$ of degree $d$. Then $f$ admits a {\it linear Pfaffian $R$-representation} if there exist $2d\times2d$ skew-symmetric matrices $A_0,\dots,A_n$ with
entries in $R$ such that
$$
\Pf(x_0A_0+\cdots+x_nA_n)= f(x_0,\dots,x_n).
$$
\end{definition}
\noindent Such a matrix $M=x_0A_0+\cdots+x_nA_n$ is said to be a linear Pfaffian $R$-representation of $f$.

\begin{remark}
Example \ref{ejemplo: grado 2 y 3} shows that any form $f(x,y,z)$ in $R[x,y,z]$ of degree $d\leq3$ admits a linear Pfaffian $R$-representation. $M_2$ and $M_3$ are
{\it explicit} linear Pfaffian $R$-representations. % in the corresponding degrees.
\end{remark}

%%%%%%%%%%%%%%%%%%%%%%%%%%%%%%%%%%%%%%%%%%%%%%%%%%%%%%%%%%%%%%%%%%%%%%%%%%%%%%%%%%%%%%%%%%%%%%%%%%%
%%%%%%%%%%%%%%%%%%%%%%%%%%%%%%%%%%%%%%%%%%%%%%%%%%%%%%%%%%%%%%%%%%%%%%%%%%%%%%%%%%%%%%%%%%%%%%%%%%%
%%%%%%%%%%%%%%%%%%%%%%%%%%%%%%%%%%%%%%%%%     SECTION                     %%%%%%%%%%%%%%%%%%%%%%%%%
%%%%%%%%%%%%%%%%%%%%%%%%%%%%%%%%%%%%%%%%%%%%%%%%%%%%%%%%%%%%%%%%%%%%%%%%%%%%%%%%%%%%%%%%%%%%%%%%%%%
%%%%%%%%%%%%%%%%%%%%%%%%%%%%%%%%%%%%%%%%%%%%%%%%%%%%%%%%%%%%%%%%%%%%%%%%%%%%%%%%%%%%%%%%%%%%%%%%%%%

\section{Proof of the Theorem \ref{Grado 4 y 5}}

\subsection{Restricted search}
Notice that in the Example \ref{ejemplo: grado 2 y 3} the linear Pfaffian $R$-representations $M_2$ and $M_3$ are very simple: They only have $0$'s, $\pm1$'s and the
coefficients of an arbitrary homogeneous polynomial. With this in mind we introduce:
\begin{definition}\label{nice}A {\it nice} representation $M=\varuno\, [a_{ij}]+\vardos\, [b_{ij}]+\vartres\, [c_{ij}]$ of $f(x,y,z)$ is a Pfaffian $R$-representation such that:
\begin{itemize}
  \item[\textbullet] $a_{ij}, b_{ij}, c_{ij}\in\{-1,0,1\} \cup \{\pm\textrm{coefficients of }f(x,y,z)\textrm{ of degree }d\}$.
  \item[\textbullet] For each coefficient $\Theta_i$ of $f(x,y,z)$, there exists a unique entry $e_{ij}\in\{a_{ij}, b_{ij}, c_{ij}\}$ such that $e_{ij}=\pm\Theta_i$\,.
\end{itemize}
\end{definition}

For degrees $d\leq4$, we get a {\it nice} global representation. We also find a global representation for degree 5 (eqns. (\ref{m_5: parte 1}) and (\ref{m_5: parte 2})),
but it is not {\it nice}. We do not know if  such a {\it nice} global representation always exists.

\begin{proof}[Proof of Theorem \ref{Grado 4 y 5}] We must prove that for any
\begin{align*}
f(\varuno,\vardos,\vartres) &= \variable_1 \espacio \varuno^4 +\variable_2 \espacio \vardos^4+\variable_3 \espacio \vartres^4+
          +\variable_4 \espacio \varuno^3\vardos +\variable_5 \espacio \varuno^2\vardos^2+\variable_6 \espacio \varuno\vardos^3+
      +\variable_7 \espacio \varuno^3\vartres +\variable_8 \espacio \varuno^2\vartres^2+ \nonumber \\
      & \qquad   + \variable_9 \espacio \varuno\vartres^3+
        +\variable_{10} \espacio \vardos^3\vartres +\variable_{11} \espacio \vardos^2\vartres^2+\variable_{12} \espacio \vardos\vartres^3+
           +\variable_{13} \espacio \varuno^2\vardos\vartres +\variable_{14} \espacio \varuno\vardos^2\vartres+\variable_{15} \espacio \varuno\vardos\vartres^2\,,\nonumber \\
      %%%%%%%%%%%%%%%%%%%%%%%%%%
 g(\varuno,\vardos,\vartres) &= \variabledos_1 \espacio \varuno^5 +
                               \variabledos_2 \espacio \vardos^5+
                               \variabledos_3 \espacio \vartres^5+
      +\variabledos_{4} \espacio \varuno^4\vardos +
                \variabledos_{5} \espacio \varuno^3\vardos^2+
                \variabledos_{6} \espacio \varuno^2\vardos^3+
                \variabledos_{7} \espacio \varuno   \vardos^4+\nonumber \\
      & \qquad +\variabledos_{8} \espacio \varuno^4\vartres +
                \variabledos_{9} \espacio \varuno^3\vartres^2+
                \variabledos_{10} \espacio \varuno^2\vartres^3+
                \variabledos_{11} \espacio \varuno   \vartres^4+
   +\variabledos_{12} \espacio \vardos^4\vartres +
                \variabledos_{13} \espacio \vardos^3\vartres^2+
                \variabledos_{14} \espacio \vardos^2\vartres^3+\nonumber \\
      & \qquad
               + \variabledos_{15} \espacio \vardos   \vartres^4+ \variabledos_{16} \espacio \varuno^3\vardos\vartres +
                 \variabledos_{17} \espacio \varuno\vardos^3\vartres +
                 \variabledos_{18} \espacio \varuno\vardos\vartres^3
     + \variabledos_{19} \espacio \varuno^2\vardos^2\vartres +
                 \variabledos_{20} \espacio \varuno^2\vardos\vartres^2 +
                 \variabledos_{21} \espacio \varuno\vardos^2\vartres^2\,, \nonumber
\end{align*}
where $\variable_1,\dots,\variable_{15},\variabledos_1,\dots,\variabledos_{21}\in R$, there exist $8\times8$ skew-symmetric matrices $[a_{ij}], [b_{ij}]$ and
$[c_{ij}]$,  respectively, and $10\times10$ skew-symmetric matrices $\big[\tilde{a}_{ij}\big], \big[\tilde{b}_{ij}\big], \big[\tilde{c}_{ij}\big]$ with entries in $R$ such that
\begin{align*}
\textrm{\emph{Pf}}(\varuno\, \big[a_{ij}\big]+\vardos\, \big[b_{ij}\big]+\vartres\, \big[c_{ij}\big]) &=f(\varuno,\vardos,\vartres)\;, \\
\textrm{\emph{Pf}}(\varuno\, \big[\tilde{a}_{ij}\big]+\vardos\, \big[\tilde{b}_{ij}\big]+\vartres\, \big[\tilde{c}_{ij}\big]) &=g(\varuno,\vardos,\vartres)\;.
\end{align*}

Let $M_f=\varuno\, [a_{ij}]+\vardos\, [b_{ij}]+\vartres\, [c_{ij}]$ be the $8\times8$ skew-symmetric matrix, where
\begin{equation}\label{m_4: parte 1}
[a_{ij}]=  \left[
\begin {smallmatrix}
%%%%%%%%%%%%%%%%%%%%%%%%%%%%%
0&\Theta_{{1}}&\Theta_{{6}}&0&\Theta_{{4}}&0&0&0\\
\noalign{\medskip} &0&0&0&0&0&0&0\\
\noalign{\medskip} & &0&1&\Theta_{15}&0&0&\Theta_9\\
\noalign{\medskip}&&&0&0&0&0&0\\
\noalign{\medskip}&&&&0&1&0&\Theta_8\\
\noalign{\medskip}&&*&&&0&0&0\\
\noalign{\medskip}&&&&&&0&-1\\
\noalign{\medskip}&&&&&&&0
%%%%%%%%%%%%%%%%%%%%%%%%%%%%%
\end {smallmatrix} \right],
[b_{ij}]=  \left[
\begin {smallmatrix}
%%%%%%%%%%%%%%%%%%%%%%%%%%%%%
0&0&\Theta_{{2}}&0&\Theta_{5}&0&0&0\\
\noalign{\medskip}&0&0&0&0&-1&0&\Theta_{14}\\
\noalign{\medskip}&&0&0&\Theta_{11}&0&0&0\\
\noalign{\medskip}&&&0&0&0&0&-1\\
\noalign{\medskip}&&&&0&0&1&0\\
\noalign{\medskip}&&*&&&0&0&0\\
\noalign{\medskip}&&&&&&0&0\\
\noalign{\medskip}&&&&&&&0
%%%%%%%%%%%%%%%%%%%%%%%%%%%%%
\end {smallmatrix} \right],
[c_{ij}]=  \left[
\begin {smallmatrix}
%%%%%%%%%%%%%%%%%%%%%%%%%%%%%
0&\Theta_{7}&\Theta_{10}&0&\Theta_{13}&-1&0&0\\
\noalign{\medskip}&0&0&0&0&0&-1&0\\
\noalign{\medskip}&&0&0&\Theta_{12}&0&0&\Theta_{3}\\
\noalign{\medskip}&&&0&1&0&0&0\\
\noalign{\medskip}&&&&0&0&0&0\\
\noalign{\medskip}&&*&&&0&0&0\\
\noalign{\medskip}&&&&&&0&0\\
\noalign{\medskip}&&&&&&&0
%%%%%%%%%%%%%%%%%%%%%%%%%%%%%
\end {smallmatrix} \right].
\end{equation}

Using Definition \ref{definicion recursiva}, by a straightforward computation,  one gets that $\textrm{Pf}(M_f)=f(x,y,z)$. % and $\textrm{Pf}(M_g)=g(x,y,z)$.
%Also it can impplement that recursive definition using e.g. Maplesoft\texttrademark.

It remains to find some linear Pfaffian $R$-representation of $g(x,y,z)$.

Let $M_g=\big[\tilde{a}_{ij}\big]+\vardos\, \big[\tilde{b}_{ij}\big]+\vartres\, \big[\tilde{c}_{ij}\big]$  be the $10\times10$ skew-symmetric matrix, where
\begin{align}\label{m_5: parte 1}
\big[\tilde{a}_{ij}\big] &
=\left[
\begin {smallmatrix}
%%%%%%%%%%%%%%%%%%%%%%%%%%%%%
0&\Phi_{1}&0&0&0&0&0&0&0&0\\
\noalign{\medskip}&0&a_{2,3}&0&\Phi_{7}&0&\Phi_{5}&-\Phi_{9}&a_{2,9}&0\\
\noalign{\medskip}&&0&1&0&0&0&0&a_{3,9}&0\\
\noalign{\medskip}&&&0&0&0&0&0&0&0\\
\noalign{\medskip}&&&&0&1&0&0&0&0\\
\noalign{\medskip}&&&&&0&0&0&0&0\\
\noalign{\medskip}&&&&&&0&1&0&0\\
\noalign{\medskip}&&*&&&&&0&0&0\\
\noalign{\medskip}&&&&&&&&0&1\\
\noalign{\medskip}&&&&&&&&&0
%%%%%%%%%%%%%%%%%%%%%%%%%%%%%
\end {smallmatrix} \right] \textrm{\,,}&
\big[\tilde{b}_{ij}\big] &
=\left[
\begin {smallmatrix}
%%%%%%%%%%%%%%%%%%%%%%%%%%%%%
0&\Phi_{{4}}&1&0&0&0&0&0&0&0\\
\noalign{\medskip}&0&b_{2,3}&0&\Phi_{2}&0&\Phi_{6}&0&b_{2,9}&0\\
\noalign{\medskip}&&0&0&0&0&0&\Phi_{15}&b_{3,9}&0\\
\noalign{\medskip}&&&0&0&0&0&1&0&0\\
\noalign{\medskip}&&&&0&0&0&0&0&0\\
\noalign{\medskip}&&&&&0&0&b_{6,8}&b_{6,9}&1\\
\noalign{\medskip}&&&&&&0&0&1&0\\
\noalign{\medskip}&&*&&&&&0&0&0\\
\noalign{\medskip}&&&&&&&&0&0\\
\noalign{\medskip}&&&&&&&&&0
%%%%%%%%%%%%%%%%%%%%%%%%%%%%%
\end {smallmatrix} \right]\textrm{\,,}
\end{align}

\begin{align}\label{m_5: parte 2}
%%%%%%%%%%%%%%%%%%%%%%%%%%%%%%%%%%%%%%%%%%
%%%%%%%%%%%%%%%%%%%%%%%%%%%%%%%%%%%%%%%%%%
%%%%%%%%%%%%%%%%%%%%%%%%%%%%%%%%%%%%%%%%%%
\big[\tilde{c}_{ij}\big] &
=\left[
\begin {smallmatrix}
%%%%%%%%%%%%%%%%%%%%%%%%%%%%%
0&\Phi_{{8}}&0&0&1&0&0&0&0&0\\
\noalign{\medskip}&0&c_{2,3}&0&0&0&0&0&0&1\\
\noalign{\medskip}&&0&0&0&0&0&\Phi_{3}&0&0\\
\noalign{\medskip}&&&0&0&0&0&0&1&0\\
\noalign{\medskip}&&&&0&0&0&0&0&0\\
\noalign{\medskip}&&&&&0&1&0&-\Phi_{10}&0\\
\noalign{\medskip}&&&&&&0&0&0&0\\
\noalign{\medskip}&&*&&&&&0&-\Phi_{11}&0\\
\noalign{\medskip}&&&&&&&&0&0\\
\noalign{\medskip}&&&&&&&&&0
%%%%%%%%%%%%%%%%%%%%%%%%%%%%%
\end {smallmatrix} \right]\,,
\end{align}
and where
\begin{align*}
a_{2,3}  & =   -2\,\Phi_{{6}}\Phi_{{5}}\Phi_{{3}}-\Phi_{{6}}\Phi_{{18}}-\Phi _{{5}}\Phi_{{14}}+1-\Phi_{{17}}\,,       \\
a_{2,9}  & =    -\Phi_{{6}}\Phi_{{5}}\Phi_{{11}}-\Phi_{{15}}{\Phi_{{5}}}^{2}- \Phi_{{5}}\Phi_{{21}}+\Phi_{{16}}\,,    \\
a_{3,9}  & =  -\Phi_{{5}}\Phi_{{3}}-\Phi_{{18}}\,,      \\
             %\end{align*}
             %\begin{align*}
b_{2,3}  & =   -{\Phi_{{6}}}^{2}\Phi_{{3}}-\Phi_{{12}}-\Phi_{{6}}\Phi_{{14}}\,,   \\
b_{2,9}  & =    {\Phi_{{5}}}^{2}\Phi_{{3}}+\Phi_{{5}}\Phi_{{18}}-{\Phi_{{6}}}^{2}\Phi_{{11}}-\Phi_{{6}}\Phi_{{15}}\Phi_{{5}}-\Phi_{{6}}\Phi_{{21}}-
\Phi_{{9}}+\Phi_{{19}}\,,   \\
b_{3,9}  & =   -\Phi_{{6}}\Phi_{{3}}-\Phi_{{14}}\,,   \\
b_{6,8}  & =   \Phi_{{6}}\Phi_{{11}}+\Phi_{{15}}\Phi_{{5}}+\Phi_{{21}}\,,   \\
b_{6,9}  & =   -\Phi_{{5}}\Phi_{{11}}-1-\Phi_{{20}}\,.   \\
c_{2,3}  & =   -\Phi_{6}\Phi_{15}-\Phi_{13}\,.   \\
\end{align*}
Similarly, we get that $\textrm{Pf}(M_g)=g(x,y,z)$, which finishes the proof of the theorem.
\end{proof}

%%%%%%%%%%%%%%%%%%%%%%%%%%%%%%%%%%%%%%%%%%%%%%%%%%%%%%%%%%%%%%%%%%%%%%%%%%%%%%%%%%%%%%%%%%%%%%%%%     \appendix
%%%%%%%%%%%%%%%%%%%%%%%%%%%%%%%%%%%%%%%%%%%%%%%%%%%%%%%%%%%%%%%%%%%%%%%%%%%%%%%%%%%%%%%%%%%%%%%%%     \section{}\label{aped.A}

%%%%%%%%%%%%%%%%%%%%%%%%%%%%%%%%%%%%%%%%%%%%%%%%%%%%%%%%%%%%%%%%%%%%%%%%%%%%%%%%%%%%%%%%%%%%%%%%%%%
%%%%%%%%%%%%%%%%%%%%%%%%%%%%%%%%%%%%%%%%%%%%%%%%%%%%%%%%%%%%%%%%%%%%%%%%%%%%%%%%%%%%%%%%%%%%%%%%%%%
%%%%%%%%%%%%%%%%%%%%%%%%%%%%%%%%%%%%%%%%%     SECTION                     %%%%%%%%%%%%%%%%%%%%%%%%%
%%%%%%%%%%%%%%%%%%%%%%%%%%%%%%%%%%%%%%%%%%%%%%%%%%%%%%%%%%%%%%%%%%%%%%%%%%%%%%%%%%%%%%%%%%%%%%%%%%%
%%%%%%%%%%%%%%%%%%%%%%%%%%%%%%%%%%%%%%%%%%%%%%%%%%%%%%%%%%%%%%%%%%%%%%%%%%%%%%%%%%%%%%%%%%%%%%%%%%%

%\section{Sort}

\end{document}